\documentclass{article}
\usepackage[utf8]{inputenc}
\usepackage[noblocks]{authblk}

\usepackage{amsmath,amssymb,amsthm,bbm}
\usepackage{mathtools}
\usepackage{amsfonts}
\usepackage{enumerate}
\usepackage{authblk}
\usepackage{xcolor}
\usepackage[margin=1in]{geometry}
\theoremstyle{plain}

\newtheorem{theorem}{Theorem}[section]
\newtheorem{lemma}[theorem]{Lemma}
\newtheorem{prop}[theorem]{Proposition}
\newtheorem{coro}[theorem]{Corollary}
\theoremstyle{remark}

\newtheorem{remark}{Remark}
\newcommand\nc\newcommand
\nc{\bs}{\boldsymbol}

\nc{\PP}{\mathbb{P}}
\newcommand{\R}{\mathbb R}
\nc{\ind}{\mathbbm{1}}
\newcommand{\E}{\mathbb E}

\newcommand{\Var}{\mathrm{Var}}

\newcommand{\tr}{\mathrm{tr}}

\newcommand{\comment}[1]{}

\DeclarePairedDelimiter{\abs}{\lvert}{\rvert}
\DeclarePairedDelimiter\norm{\lVert}{\rVert}

\title{On measures strongly log-concave on a subspace}
\author{Pierre Bizeul }
\date{}
\begin{document}
\maketitle
\section{Introduction}
Let $ V:\R^n\mapsto \R$ be a $C^2$ convex function, such that $d\mu(x) = e^{-V(x)}dx$ is a log-concave probability measure. It is well-known that if $\mu$ is $t$-strongly log-concave, that is $V$ satisfies the Bakry-Émery condition :  \begin{equation}{\label{eq:criterion}}
    \nabla^2 V\geq t
\end{equation} for some $t>0$, where $\nabla^2$ stands for the Hessian, it has good isoperimetric properties. In particular, its Poincaré constant is at most $\frac{1}{t}$. Recall that the measure $\mu$ is said to satisfy a Poincaré inequality with constant $c$ if for all locally Lipschitz function $f$ we have :

$$ Var_\mu(f) \leq c^2\,\E_\mu(\abs{\nabla f}^2) $$
where here and in the sequel, $\abs{.}$ stands for the Euclidean norm. The best such constant is denoted by $c_P(\mu)$, the Poincaré constant of $\mu$. The KLS conjecture \cite{KLS1995} proposes that when $\mu$ is log-concave, its Poincaré constant is, up to a universal constant, less than the operator norm of its covariance matrix. Since Poincaré inequalities are homogeneous, we can state the conjecture only for normalized measures, without loss of generality. A measure $\mu$ is called isotropic if it is centered and its covariance matrix is the identity. Introduce
$$ \Psi_n = \sup_{\substack{\mu}} \ c_P(\mu), $$
where the supremum runs over all isotropic log-concave measures of $\R^n$. The KLS conjecture then reads :
$$\Psi_n \leq c $$
for some universal constant $c>0$. 

A related property of strongly log-concave probabilities is that they exhibit good concentration function. Recall that the concentration function of a measure $\mu$ is the function $\alpha_{\mu} : \R^+\mapsto \left[0,1/2\right]$ defined by :
$$\alpha_\mu(r) = \sup_{\{S,\, \mu(S) = 1/2\}} \mu(S_r^c) $$
where $S_r = \{x\in\R^n \, / \, d(x,S) \leq r\}$ and $d(x,S)$ is the Euclidean distance between $x$ and $S$. It follows from the Prékopa-Leindler inequality that if $\mu$ is $t$-strongly log-concave, then for all measurable sets $S$, 
\begin{equation}\label{gaussian_concentration_general}
    \mu(S_r^c) \leq \frac{1}{\mu(S)}\exp\left(-\frac{tr^2}{4}\right).
\end{equation}
In particular, it has a Gaussian-type concentration function :
\begin{equation}\label{eq:gaussian_concentration_function}
    \alpha_\mu(r) \leq 2\exp\left(-\frac{tr^2}{4}\right),
\end{equation}
see for instance \cite{Eldan2013} Proposition 2.6 and its proof. It was first observed by Gromov and Milman (\cite{gromilman} see also \cite{ledoux2001concentration} Corollary 3.2 for a better constant) that a Poincaré inequality implies exponential concentration, that is :
\begin{equation}\label{eq:grom_exp}
    \alpha_\mu(r) \leq \exp\left({-\frac{r}{3c_P(\mu)}}\right).
\end{equation}
The converse implication has been established in the log-concave case by E.Milman \cite{Milman2009} where he shows that when $\mu$ is log-concave, 
\begin{equation}\label{eq:milman_exp}
    c_P(\mu) \lesssim \alpha_\mu^{-1}(1/4)
\end{equation}
where for two expressions $a,b$ depending on parameters, $a\lesssim b$ means there is a universal constant $c>0$ such that $a \leq cb$. We also write $a\simeq b$ when $a\lesssim b$ and $b\lesssim a$.

In an attempt to tackle the KLS conjecture, Eldan \cite{Eldan2013} introduced a stochastic process, known as stochastic localization, which, roughly, decomposes, for all time $t\geq0$, a log-concave measure $\mu$ into an average of measures $\mu_t(\omega)$ which are $t$-strongly log-concave. This strategy enabled Eldan to relate the KLS conjecture to the a priori weaker Variance conjecture, then Lee and Vempala \cite{leevamp} to obtain the then better bound on $\Psi_n$ : $\Psi_n\lesssim n^{1/4}$. Recently, Chen obtained that $\Psi_n = o(n^\alpha)$ for every $\alpha>0$, \cite{Chen2021}, and very recently, Klartag and Lehec obtained $\Psi_n = O(\log(n)^{5})$ \cite{klarleh}.

In this note, we propose a slight generalization of the criterion (\ref{eq:criterion}) allowing the potential to be flat in some directions. The observation is that stochastic localization behaves well when restrained to a subspace.

Our main result is the following :
\begin{theorem}\label{thm:main} Let $V:\R^n\mapsto\R$ be a $C^2$ convex potential such that $d\mu(x) = e^{-V(x)}dx$ is a probability measure. Suppose that there is $1\leq k \leq n$, a subspace $E$ of codimension $k$ and $\eta>0$ such that 
$$\nabla^2V \geq \eta P_E $$
where $P_E$ is the orthogonal projector onto E. Let $K$ be the covariance matrix of $\mu$. Define $Q = P_{E^\perp} KP_{E^\perp}$
\begin{enumerate}[(i)]
    \item $ c_P(\mu) \lesssim \max\left(\frac{1}{\sqrt{\eta}}\,,\, \norm{Q}_{op}^{1/2}\Psi_k\sqrt{\max(\log(k),1)}\right)$
    \item There is a universal constant $c>0$ such that for every $A$ such that $\mu(A) = \frac{1}{2}$, 
    
        $$\mu(A_r^c)\lesssim \exp\left(-c\min\left(\frac{r}{\norm{Q}_{op}^{1/2}}\,,\, r^2\min\left(\eta,\frac{1}{\Psi_k^2\max(\log(k),1)\norm{Q}_{op}}\right)\right)\right) $$
\end{enumerate}
\end{theorem}

In the particular case $E=\{0\}$, inequality \textit{(ii)} implies a new bound for the concentration function of log-concave measures, which we state, without loss of generality, in the isotropic case.

\begin{coro}\label{coro:concentration} For any isotropic log-concave measure $\mu$ and any $r>0$, we have 

$$\alpha_\mu(r) \lesssim \exp\left(-c\min\left(r,\frac{r^2}{\psi_n^2\log(n)}\right)\right) $$
\end{coro}

\begin{remark}
Note that \textit{(ii)} implies \textit{(i)}. Indeed, choosing $r = c'\,\max\left(\frac{1}{\sqrt{\eta}}\,,\,\norm{Q}_{op}^{1/2}\Psi_k\sqrt{\max(\log(k),1)}\right)$, for an appropriate choice of constant $c'>0$, we get that $\mu(A_r^c) \leq \frac{1}{4}$. By (\ref{eq:milman_exp}), this implies \textit{(i)}. On the other hand it is easy to check that the exponential concentration obtained by combining \textit{(i)} with (\ref{eq:grom_exp}) is weaker than \textit{(ii)}.
\end{remark}
\begin{remark}
The idea of evaluating concentration functions with stochastic localization already appears in the work of Lee and Vempala (\cite{leevamp}, Theorem 16). To improve the Paouris deviation inequality for the Euclidean norm (\cite{Paouris2006}),  they develop a more refined analysis of the process, using the so-called  Stieltjes potential. They prove that for any $L$-Lipschitz function $g$, and any isotropic log-concave probability measure $\mu$ one has :
\begin{equation}\label{eq:lip-concentration}
    \forall t\geq 0 \quad \PP(\vert g(X) - \Bar{g}(X)\vert \geq Lt) \leq \exp\left(-\frac{ct^2}{t+\sqrt{n}}\right)
\end{equation}
where $X\sim\mu$ and $\Bar{g}(X)$ is the median or mean of $g(X)$. Notice that when $g$ is the Euclidean norm, then by Borell's Lemma \cite{borell}, $\E_\mu(\abs{x}) \simeq \E_\mu(\abs{x}^2)^{1/2} = \sqrt{n}$ since $\mu$ is isotropic. Plugging this into \eqref{eq:lip-concentration} yields
\begin{equation}\label{eq:paouris}
    \forall t\geq 0 \quad \PP(\abs{X} \geq t\sqrt n) \leq \exp(-c\min(t,t^2)\sqrt n)
\end{equation}

However, thanks to the new estimate of Chen, $\Psi_n = o(n^\alpha)$ for every $\alpha>0$, we can obtain this result directly from Corollary \ref{coro:concentration}. Indeed, for a general isotropic log-concave probability measure $\mu$ it asserts that for all measurable $A$ such that $\mu(A)=1/2$ and all $r>0$,

$$\mu(A_r^c) \lesssim \exp\left(-c\min\left(r,\frac{r^2}{\psi_n^2\log(n)}\right)\right). $$

Let $g$ be a $L$-Lipschitz function, and let  $A = \{x\in\R^n, \ g(x) \leq \Bar{g}(X) \}$, by definition of the median, $\mu(A)=1/2$. Now set $G_r = \{ x\in\R^n, \ g(x) \leq \Bar{g}(X) + Lr\}$, then because $g$ is $L$-Lipschitz, $ A_r \subset G_r$, where $A_r$ is the $r$-extension of $A$. We get that 
$$\mu(G_r^c) \leq \mu(A_r^c) \lesssim \exp\left(-c\min\left(r,\frac{r^2}{\psi_n^2\log(n)}\right)\right), $$

For the Euclidean norm, which is $1$-Lipschitz, this yields
$$\PP(\abs{X} \geq r\sqrt{n}) \lesssim \exp\left(-c\min\left(r\sqrt{n},\frac{r^2n}{\psi_n^2\log(n)}\right)\right) \lesssim \exp\left(-c\min(r,r^2)\sqrt{n}\right)$$
where we used the fact that $\psi_n^2\log(n) = o(\sqrt n)$ thanks to Chen's estimate. Notice that using the Lee-Vampala estimate $\psi_n^2 = O(\sqrt n)$ would lead to an extra logarithmic factor in the deviation estimate whose removal was the object of their work with the Stieltjes potential.
\end{remark}
\begin{lemma}
It is enough to prove Theorem \ref{thm:main} when $Q = I_k$.
\end{lemma}
\begin{proof}
Let $d\mu(x) = e^{-V(x)}dx$ be a measure satisfying the hypothesis of Theorem \ref{thm:main} and let $X$ be a random vector whose law is $\mu$. Set $S = \begin{bmatrix}
\norm{Q}_{op}^{1/2}I_{n-k} & 0\\
0 & Q^{1/2} 
\end{bmatrix}$ where the matrix is expressed in a basis adapted to the splitting $\R^n = E \oplus E^\perp$. Define the random vector $\tilde{X} = S^{-1}X$, whose law is $d\tilde{\mu}(x) = e^{-\tilde{V}(x)}dx = \vert \det S\vert \ e^{-V(Sx)}dx$ and covariance matrix  $$\tilde{K} = S^{-1}KS^{-1}.$$ For a symmetric $n\times n$ matrix $M$, we denote by $\lambda_1(M)\geq \dots\geq\lambda_n(M)$ its ordered eigenvalues. It is classical and easy to check that for every $r>0$ one has
\begin{equation}\label{eq:renormalisation}
    \alpha_{\mu}(r) \leq \alpha_{\tilde\mu}\left(\frac{r}{\lambda_1(S)}\right) = \alpha_{\tilde\mu}\left(\frac{r}{\norm{Q}_{op}^{1/2}}\right)
\end{equation}
However, with this choice of $S$, $\tilde{\mu}$ satisfies :

$$\lambda_{n-k}(P_E\nabla^2\tilde{V}P_E) \geq \tilde{\eta} = \norm{Q}_{op}\eta \quad \text{and}\quad \tilde Q = P_{E^\perp}\tilde KP_{E^\perp} = I_k. $$
We can then apply Theorem \ref{thm:main} to $\tilde{\mu}$ which, combined with (\ref{eq:renormalisation}), yields the result.
\end{proof}

We conclude this introduction with a classical inequality, which essentially goes back to Freedman \cite{freedman1975tail}, that we will use for controlling deviation of martingales in the sequel.

\begin{lemma}\label{lem:bernstein}
Let $M_t$ be a continuous local martingale starting from 0.
$$\forall T >0 \quad \PP(M_T \geq a \ , \ [M]_T\leq b) \leq   \exp(-\frac{a^2}{2b})$$
\end{lemma}
\begin{proof}
For all $\lambda \in \R$, define the process $\mathcal{E}(\lambda M)$ by 
$$\mathcal{E}(\lambda M)_t  = \exp\left(\lambda M_t - \frac{\lambda^2}{2}[M]_t\right).$$
Elementary Itô calculus shows that $\mathcal{E}(\lambda M)$ is a local martingale. Moreover, it is positive, so by Fatou's lemma it is a supermartingale. In particular, for all $t\geq0$, $\E \ \mathcal{E}(\lambda M)_t \leq \mathcal{E}(\lambda M)_0 = 1$, that is :
$$ \forall t \geq 0, \quad\E \exp\left(\lambda M_t -\frac{\lambda^2}{2}[M]_t\ \right) \leq 1$$

Now, assume that  $[M]_T \leq b$ almost surely. Then,

\begin{align*}
    \PP(M_T \geq a ) &= \PP(  \mathcal{E}(\lambda M)_t \geq  e^{\lambda a - \frac{\lambda^2}{2}[M]_T})\\
    &\leq \PP(  \mathcal{E}(\lambda M)_t \geq  e^{\lambda a - \frac{\lambda^2}{2}b})\\
    &\leq \E( \mathcal{E}(\lambda M)_t)\ e^{\frac{\lambda^2}{2}b-\lambda a} \\
    &\leq e^{\frac{\lambda^2}{2}b-\lambda a}\\
\end{align*}
Choosing the optimal $\lambda = \frac{a}{b}$ yields :
$$ \PP(M_T \geq a ) \leq e^{-\frac{a^2}{2b}}.$$

The proof follows from applying this argument to the local martingale $M_t^\tau = M_{t\wedge \tau}$, where $\tau = \inf \{t \geq 0, [M]_t \geq b\}$ is a stopping time. Indeed, remark that $[M_{t\wedge \tau}]_t \leq b$ almost surely, and that 

$$\PP(M_T \geq a \ , \ [M]_T\leq b) \leq \PP(M_T^\tau \geq a) \leq e^{-\frac{a^2}{2b}}.$$ 

\end{proof}
\section{Restricted stochastic localization}

Let $\mu$ be a log-concave measure satisfying the hypothesis of Theorem \ref{thm:main} with $Q=I_k$. We denote by $P:\R^n\mapsto\R^n$ the orthogonal projection onto the $k$-dimensional subspace $E^\perp$. In the following we work in an orthonormal basis such that this subspace is spanned by the $k$ first basis vectors. Let $f$ be the density of $\mu$, for all $x\in\R^n$, consider the following stochastic differential equations :

\begin{equation}\label{eq:process}
    df_t(x) = (x-a_t)^TPdB_tf_t(x) \quad ; \quad f_0(x) = f(x)
\end{equation}
where $a_t = \int_{\R^n} xf_t(x)dx$ is the barycenter of the measure $\mu_t$, which we define here as having density $f_t$, and $(B_t)_{t\geq0}$ is a standard Brownian motion on $\R^n$.

This system of equation is the same as the usual stochastic localization, except for the addition of the matrix $P$ which projects the random direction given by the Brownian onto the subspace where we need to bend the potential. The idea of adding a projector first appears in a paper of Klartag \cite{Klartag2018} for other purposes. The following facts and computations are very standard, and we refer the reader to \cite{Eldan2013} and \cite{leevamp} for a more detailed exposition. In particular, we need to assume that the support of $\mu$ is bounded to grant the existence and well-definedness of the process for all time $t\geq0$ and then extend the result to arbitrary $\mu$ by approximation; we again refer to \cite{Eldan2013}.

\begin{prop} 
\begin{itemize}
    \item Equation (\ref{eq:process}) defines a function-valued martingale $f_t$ in the sense that for any continuous and compactly supported function $\phi$ :
    
    \begin{equation}\label{eq:martingale}
   \int_{\R^n}\phi(x) f_t(x)dx \quad \text{is a martingale}
    \end{equation}
    \item $f_t$ is a density and for all $x\in\R^n$,
    \begin{equation}\label{eq:ft_expression}
        f_t(x) = \frac{1}{Z_t}e^{-\frac{t}{2}x^TPx + c_t\cdot x}f(x):=e^{-V_t(x)} 
    \end{equation}
    where $c_t$ is the solution of :
    \begin{equation}\label{eq:c_t}
            c_0 = 0, \quad dc_t = PdB_t + Pa_tdt
    \end{equation}
    in particular we see that $\nabla^2 V_t \geq \min(\eta,t)Id$
\end{itemize}
\end{prop}
\begin{proof}
For the existence and well-definedness of the process, see the remark below. While it is possible to check that $f_t$ as defined by \eqref{eq:ft_expression} satisfy \eqref{eq:process}, we sketch a different proof to lighten the exposition. Let $m_t = \int_{\R^n}f_t(x)dx$ be the total mass at time $t$. Recall that$a_t = \frac{1}{m_t}\int_{\R^n}xf_t(x)dx$ is the barycenter of $f_t$. Then, by \eqref{eq:process},
\begin{align*}
    dm_t &= \left(P\int_{\R^n}(x-a_t)f_t(x)dx\right).dB_t\\
    &= \left(Pa_t(m_t-1)\right).dB_t.
\end{align*}
It is easy to check that this simple stochastic differential equation admits a unique solution (see \cite{oksendal2003stochastic} §5.2). It is given by $m_t=1$. To establish \eqref{eq:ft_expression}, we use \eqref{eq:process} to compute:
\begin{align*}
    d\log f_t(x) &= \frac{df_t(x)}{f_t(x)}-\frac{1}{2}\frac{d[f(x)]_t}{f_t(x)^2}\\
    &= \left(P(x-a_t)\right)\cdot dB_t -\frac{1}{2}(x-a_t)^TP(x-a_t)dt\\
    &= x\cdot\left(PdB_t+Pa_tdt\right) - \frac{1}{2}x^TPx dt +dz_t\\
    &= x\cdot dc_t - \frac{1}{2}x^TPx dt + dz_t
\end{align*}
where $dz_t$ regroups the terms that do not depend on $x$. It encodes the normalizing factor $Z_t$. The expression \eqref{eq:ft_expression} together with the proof that $m_t=1$ ensures that $f_t$ is a density. The martingale property \eqref{eq:martingale} is straightforward since, for any $\phi$ compactly supported,
$$d\int_{\R^n}\phi(x) f_t(x)dx = \left(\int_{\R^n}\phi(x)P(x-a_t) f_t(x)dx\right)\cdot dB_t. $$
Finally, the lower-bound on the Hessian of $V_t$ is a direct consequence of \eqref{eq:ft_expression}.

\end{proof}

\begin{remark}
Equation \eqref{eq:process} defines an infinite system of stochastic differential equations. It is therefore a priori unclear whether a solution exists. However there is a simpler, although arguably less intuitive, way of defining the process. First notice that, given the initial data $f$, $a_t$ is but a function of $t$ and $c_t$ defined as the barycenter of the density $f_t$ \eqref{eq:ft_expression}. Hence, we can first define $c_t$ by equation (\ref{eq:c_t}) and then $f_t$ by equation (\ref{eq:ft_expression}), and only then compute $df_t(x)$.
\end{remark}

The next two lemmas are standard and straightforward computations in stochastic localization which are obtained using Itô calculus. See \cite{Eldan2013} and (\cite{leevamp},  Lemma 20). We denote by $K_t$ the covariance matrix of $\mu_t$.  Since the computation for its infinitesimal change $dK_t$ is a bit tedious, we omit it to lighten the exposition.

\begin{lemma}
$da_t = K_tPdB_t$
\end{lemma}
\begin{proof}
By Itô calculus and \eqref{eq:process},
\begin{align*}
    da_t = d\int_{\R^n}xf_t(x)dx &=\int_{\R^n}x(x-a_t)^TPf_t(x)dx\ dB_t\\
    &=\int_{\R^n}x(x-a_t)^Tf_t(x)dx\ PdB_t\\
    &=\int_{\R^n}(x-a_t)(x-a_t)^Tf_t(x)dx\ PdB_t = K_tPdB_t
\end{align*}
\end{proof}

\begin{lemma}
$dK_t = \int_{\R^n} (x-a_t)(x-a_t)^TP(x-a_t)^TdB_tf_t(x)dx - K_tPK_tdt$
\end{lemma}

Now  we  want  to  have  an  estimate  of  the  concentration function of $\mu$. We first need to understand how the measure of a set evolves along the process.

\begin{lemma}\label{lem:expansion}
Let $S\subset\R^n$ be a measurable set and define $s_t = \mu_t(S)$, then :

$$d[s]_t \leq (\norm{PK_tP}_{op}) dt $$

\end{lemma}

\begin{proof}

\begin{align*}
    ds_t = \int_S df_t(x)dx = \langle\int_S P(x-a_t)f_t(x)dx , dB_t \rangle
\end{align*}
So the quadratic variation is
\begin{align*}
    d[s]_t &= \max_{ \abs{\xi}\leq1} \left(\int_S \xi^TP(x-a_t)f_t(x)dx \right)^2 dt\\
    &\leq \max_{ \abs{\xi}\leq1}\left( \int_S \left(\xi^TP(x-a_t)\right)^2f_t(x)dx\right) \left(\int_S f_t(x)dx\right) dt \\
    & \leq \max_{\abs{\xi}\leq1} (\xi^TPK_tP\xi) dt\leq (\norm{PK_tP}_{op}) dt
\end{align*}
\end{proof}

To control the above quadratic variation, we need to control the norm of $Q_t = PK_tP$. This is the purpose of the next section. 

\section{Control of the covariance matrix}

 We will see that the matrix $Q_t$, seen as a $k\times k$ matrix, follows the same dynamics as the covariance matrix of the standard stochastic localization in $\R^k$. To be more precise, it is the covariance matrix of the marginal density, which follows a stochastic localization dynamics. Hence, to control its operator norm, we use the same strategy as Eldan.

\begin{lemma} Define $g_t(y) = \int_{\R^{n-k}} f_t(y,x)dx$ the marginal density of the vector $Y_t=PX_t$, where $X_t$ is the random vector with density $f_t$. The barycenter of $Y_t$ is $b_t = Pa_t$. Then,  
\begin{equation}\label{eq:dgt_LV}
    dg_t(y) = (y-b_t)^TdW_tg_t(y)
\end{equation}
where $W_t$ is a standard Brownian in $\R^k$. Moreover $Q_t$ is the covariance matrix of $Y_t$ and 
\begin{equation}\label{eq:cov_dynamics}
    dQ_t = \int_{\R^k} (y-b_t)(y-b_t)^T(y-b_t)^TdW_tg_t(y)\ dy - Q_t^2dt
\end{equation}
\end{lemma}
\begin{proof}
The lemma follows from straightforward computations.
\end{proof}
\begin{remark}
Equation \eqref{eq:dgt_LV} is the definition of the stochastic localization process used by Lee and Vampala \cite{leevamp} and Klartag and Lehec \cite{klarleh}. It is also the process used by Chen \cite{Chen2021} when the initial measure is isotropic. Eldan \cite{Eldan2013} has a slightly different definition, even if most of the ideas used to analyze one process transfer to the other. 
\end{remark}
From now, the main purpose of this section is to show that the operator norm of $Q_t$ is bounded by a constant up to time $T = \frac{c_0}{\Psi_k^2\max(\log(k),1)}$, with $c_0$ a universal constant, see Lemma \ref{lem:control_cov} below. This result essentially goes back to Eldan \cite{Eldan2013}, in a slightly different setting, and further appears in Lee-Vampala (\cite{leevamp}, Lemma 58) and Chen (\cite{Chen2021}, Lemma 7). We provide a simplified exposition of the proof of Chen. Following Eldan, we use the potential $\Gamma_t = \tr(Q_t^p) = \sum_{i=1}^k\lambda_i^p$ for some $p\geq1$ where  $\lambda_1\geq\dots\geq\lambda_k$ are the eigenvalues of $Q_t$. In the following, we denote by $(e_1,\dots,e_k)$ a basis of eigenvectors of $Q_t$, where the dependence on $t$ and $\omega$ is implicit. 
\begin{lemma}
\begin{equation}{\label{eq:dGam}}
    d\Gamma_t = \sum_i p\lambda_i^{p-1}u_{ii}\cdot dW - \sum_i p\lambda_i^{p+1}dt + \sum_{i\neq j} p\lambda_i^{p-1}\frac{\vert u_{ij}\vert^2}{\lambda_i-\lambda_j}dt + \sum_i \frac{p(p-1)}{2}\lambda_i^{p-2}\vert u_{ii} \vert^2dt
\end{equation}
where for all $i,j$, $u_{ij} = \int_{\R^k} (y-b_t)\cdot e_i\,(y-b_t)\cdot e_j\,(y-b_t)g_t(y)dy$
\end{lemma}
\begin{proof}

The functional $\Phi : M\mapsto \tr(M^p)$ defined on symmetric matrices is $C^\infty$. On the dense open set $U$ of matrices whose eigenvalues are pairwise distinct, the functionals $M\mapsto \lambda_i(M)$ are smooth by implicit value theorem. Let $Q\in U$, with eigenvalues $\lambda_1,\dots,\lambda_k$ and eigenvectors $e_1,\dots,e_k$ and let $q_{i,j}$ be the entries of $Q$ in the basis $e$. The following computations are standard, see (\cite{eldanphd}, Lemma 1.4.8)

$$\nabla \lambda_i(Q) = e_ie_i^T \quad. $$
For the second derivative, the only non-zero terms are 
$$\frac{\partial^2\lambda_i}{\partial q_{i,j}^2} = \frac{2}{\lambda_i - \lambda_j}.$$
Combining this with \eqref{eq:cov_dynamics} proves the result when $Q_t$ belongs to $U$, it is easy to see that it extends to the general case.
\end{proof}

\begin{lemma}
\begin{equation}\label{lem:pnorm_dyn}
    d(\Gamma_t^{1/p}) = v_t\cdot dW_t + \delta_t \, dt
\end{equation}
where
\begin{equation}
    v_t = \left(\sum_i \lambda_i^p\right)^{\frac{1}{p}-1}\left(\sum_i \lambda_i^{p-1}u_{ii}\right)
\end{equation}
and
\begin{equation}\label{eq:delta_t_expression}
 \delta_t \leq (p-1)\left(\sum_i\lambda_i^p\right)^{\frac{1}{p}-1}\sum_{i,j}\lambda_i^{p-2}\vert u_{ij}\vert^2
\end{equation}
\end{lemma}
\begin{proof}
By Ito calculus, $d(\Gamma_t^{1/p}) = \frac{1}{p}\Gamma_t^{\frac{1}{p}-1}d\Gamma_t + \text{Itô term}$. But $x\rightarrow x^{1/p}$ is concave, so the Itô term is negative. Injecting equation \eqref{eq:dGam} yields $$v_t = \left(\sum_i \lambda_i^p\right)^{\frac{1}{p}-1}\left(\sum_i \lambda_i^{p-1}u_{ii}\right)$$
and 
$$ \delta_t  \leq \frac{p-1}{2}\left(\sum_i\lambda_i^p\right)^{\frac{1}{p}-1}\sum_{i}\lambda_i^{p-2}\vert u_{ii}\vert^2 +\left(\sum_i\lambda_i^p\right)^{\frac{1}{p}-1}\sum_{i\neq j}\frac{\lambda_i^{p-1}}{\lambda_i - \lambda_j }\vert u_{ij}\vert^2.$$

Now, notice that $u_{ij} = u_{ji}$ so that 
\begin{align*}
\sum_{i\neq j}\frac{\lambda_i^{p-1}}{\lambda_i - \lambda_j }\vert u_{ij}\vert^2 &= \frac{1}{2} \sum_{i\neq j}\frac{\lambda_i^{p-1} - \lambda_j^{p-1}}{\lambda_i - \lambda_j }\vert u_{ij}\vert^2\\
&\leq \frac{1}{2}\sum_{i\neq j} (p-1)\max(\lambda_i,\lambda_j)^{p-2}\vert u_{ij}\vert^2\\
&\leq \frac{p-1}{2}\sum_{i\neq j} (\lambda_i^{p-2} + \lambda_j^{p-2})\vert u_{ij}\vert^2\\
&\leq (p-1)\sum_{i\neq j}\lambda_i^{p-2}\vert u_{ij}\vert^2
\end{align*}
which proves the lemma.
\end{proof}
In the next two lemmas, we bound $\vert v_t \vert$ and $ \delta_t $ in terms of $\Gamma_t^{\frac{1}{p}}$ in order to apply a Gronwall-type argument.
\begin{lemma}\label{lem:vt_control} There is a universal constant $c>0$ such that for all $t\geq 0$, 
$$\vert v_t \vert \leq c\left(\Gamma_t^{\frac{1}{p}}\right)^{3/2} \quad \text{a.s .}$$
\end{lemma}
\begin{proof}
Let $\tilde{Y} = Y_t-b_t$ be distributed according to $g_t(y-b_t)dt$, where we drop the dependence in $t$ for readibility. Let $\tilde{Y}_1, \dots \tilde{Y}_k$ its coordinates in the basis $e_1, \dots, e_k$. $\tilde{Y}$ is  a centered log-concave vector of $\R^k$ of covariance $Q_t$ and for all $1\leq i\leq k$, $\E \tilde{Y}_i^2 = \lambda_i$. Note that for all $1\leq i\leq k$, $u_{ii} = \E\left[\tilde{Y}_i^2\tilde{Y}\right]$. Then, for all $\theta \in S^{k-1}$, 
\begin{align*}
    \vert u_{ii}\cdot\theta \vert &= \vert\E\left[\tilde{Y}_i^2\tilde{Y}\cdot\theta\right]\vert\\
    &\leq \E\left[\tilde{Y}_i^4\right]^{1/2}\E\left[(\tilde{Y}\cdot\theta)^2\right]^{1/2}\\
    &\lesssim \E\left[\tilde{Y}_i^2\right]\norm{Q_t}_{op}^{1/2}\\
    &\leq \lambda_i \Gamma_t^{1/2p}
\end{align*}
where in the second inequality we used Borell's lemma (\cite{borell}). This proves the lemma.
\end{proof}
\begin{lemma}\label{lem:deltat_control}
For all $t\geq 0$, 
\begin{align*}
\delta_t\leq 4p\Gamma_t^{2/p}\Psi_k^2\\
\end{align*}
\end{lemma}
\begin{proof}
With the same notations as in the previous lemma, for all $1\leq i,j\leq k$, $u_{ij} = \E\left[\tilde{Y}_i\tilde{Y}_j\tilde{Y}\right]$. For all $1\leq i\leq k$, we define the matrix $\Delta_i = \E\left[\tilde{Y}_i\tilde{Y}\tilde{Y}^T\right]$. Following Chen \cite{Chen2021}, we compute :

\begin{align*}
\sum_{i,j} \lambda_i^{p-2}\vert u_{ij}\vert^2 &= \sum_{i,j,k}\lambda_i^{p-2}\E(\tilde{Y}_i\tilde{Y}_j\tilde{Y}_k)^2\\
& = \sum_{i} \lambda_i^{p-2}\tr(\Delta_i^2)\\
& = \sum_{i} \lambda_i^{p-2}\tr(\Delta_i\E\tilde{Y}_i\tilde{Y}\tilde{Y}^T)\\
& = \sum_{i} \lambda_i^{p-2}\E\left(\tilde{Y}_i\tilde{Y}^T\Delta_i\tilde{Y}\right) \\
& \leq \sum_i \lambda_i^{p-2} \E(\tilde{Y}_i^2)^{1/2}\Var(\tilde{Y}^T\Delta_i\tilde{Y})^{1/2}\\
& \leq \sum_i \lambda_i^{p-2}\lambda_i^{1/2}c_P(\tilde{Y})\left(4\E\left[\vert \Delta_i\tilde{Y} \vert^2\right]\right)^{1/2}\\
& = 2c_P(\tilde{Y})\sum_i \lambda_i^{p-3/2}\tr(\Delta_iQ_t\Delta_i)^{1/2} \\
& \leq 2c_P(\tilde{Y}) \left(\sum_i \lambda_i^p\right)^{1/2}\left(\sum_i\lambda_i^{p-3}\tr(\Delta_i^2Q_t)\right)^{1/2}
\end{align*}
Now,
\begin{align*}
    \sum_i \lambda_i^{p-3}\tr(\Delta_i^2Q_t) &= \sum_i \lambda_i^{p-3}\sum_{j,k}\lambda_j(\Delta_i)_{j,k}^2\\
    & = \sum_{i,j,k}\lambda_i^{p-3}\lambda_j\E\left(\tilde{Y}_i\tilde{Y}_j\tilde{Y}_k\right)^2 \\
    & \leq \sum_{i,j,k}\lambda_i^{p-2}\E(\tilde{Y}_i\tilde{Y}_j\tilde{Y}_k)^2
\end{align*}
where in the last inequality, we used the convexity inequality : $\lambda_i^{p-3}\lambda_j \leq \frac{p-3}{p-2}\,\lambda_i^{p-2} + \frac{1}{p-2}\,\lambda_j^{p-2}$.
Plugging this into the inequality above yields :
$$\sum_{i,j} \lambda_i^{p-2}\vert u_{ij}\vert^2 \leq 2c_p(\tilde{Y})\left(\sum_i\lambda_i^p\right)^{1/2}\left(\sum_{i,j} \lambda_i^{p-2}\vert u_{ij}\vert^2\right)^{1/2}$$
which implies
$$\sum_{i,j} \lambda_i^{p-2}\vert u_{ij}\vert^2 \leq 4c_p(\tilde{Y})^2\left(\sum_i\lambda_i^p\right). $$
Plugging this into \eqref{eq:delta_t_expression} remarking that $c_P(\tilde{Y})=c_P(Y_t)$ yields :
\begin{align*}
    \delta_t &\leq 4p\Gamma_t^{1/p}c_P(Y_t)^2 \\
&\leq 4p \Gamma_t^{1/p}\norm{Q_t}_{op}\Psi_k^2\\
&\leq 4p\Gamma_t^{2/p}\Psi_k^2.\\
\end{align*}

\end{proof}
We are now in position to control the growth of $\Gamma_t$ by a Gronwall-type argument.
\begin{lemma}\label{lem:control_cov} There are constants $c_0,c_1>0$ such that for any $t\leq T = \frac{c_0}{\Psi_k^2\max(\log(k),1)}$, we have :
$$\PP\left(\max_{s\in[0,t]} \norm{Q_s}_{op} \geq 10 \right) \leq \exp{\left(-\frac{c_1}{t}\right)}. $$
As a consequence, for any measurable set $S\subset\R^n$ of measure $\mu(S) = 1/2$, setting $s_t = \mu_t(S)$, we have :

$$\PP([s]_t \geq 10t) \leq \exp{\left(-\frac{c_1}{t}\right)}$$

\end{lemma}
\begin{proof}
Set $p=\max(\log(k),1)$, so that  $$\Gamma_0^{1/p} \leq e$$
as we will use repeatedly in the proof. Recall that
    $$d(\Gamma_t^{1/p}) = v_t\cdot dW_t + \delta_t\,dt$$
and define the stopping time $\tau = \inf\{t\geq0 \ , \ \Gamma_t^{1/p} \geq 3\Gamma_0^{1/p} \}$. We denote by $M_t$ the martingale term $M_t = \int_0^tv_s\cdot dW_s$. For all $t\geq0$ we have :
\begin{align*}
    \Gamma_{t\wedge\tau}^{1/p} &= \Gamma_0^{1/p} + M_{t\wedge\tau} + \int_0^{t\wedge\tau} \delta_s ds \\
    &\leq \Gamma_0^{1/p} + \int_0^{t\wedge\tau} 4p\Gamma_s^{2/p}\Psi_k^2 ds + M_{t\wedge\tau} \hskip0.30\textwidth\left(\text{By Lemma \ref{lem:deltat_control}}\right)\\
    & \leq \Gamma_0^{1/p} + 36p\Gamma_0^{2/p}\Psi_k^2\, t + M_{t\wedge\tau} \hskip0.35\textwidth \left( \Gamma_s^{2/p} \leq 9\Gamma_0^{2/p}\right)\\
    & \leq \Gamma_0^{1/p}\left(1 + 36e\max(\log(k),1)\Psi_k^2\,t\right) + M_{t\wedge\tau} \hskip0.23\textwidth \left(\Gamma_0^{2/p}\leq e\Gamma_0^{1/p}\right)
\end{align*}
We choose $c_0 \leq \frac{1}{36e}$, so that for all $t\leq T = \frac{c_0}{\Psi_k^2\max(\log(k),1)}$,

$$ \Gamma_{t\wedge\tau}^{1/p} \leq 2\Gamma_0^{1/p} + M_{t\wedge\tau}.$$
Consequently, for all such $t$,
\begin{align*}
    \PP(\tau\leq t) &= \PP(\Gamma_{t\wedge\tau} = \Gamma_\tau)\\
    &\leq \PP\left(M_{t\wedge\tau} \geq \Gamma_0^{1/p} \right).
\end{align*}
Now, $\tau$ being a stopping time, $M_{t\wedge\tau}$ is a martingale, whose quadratic variation is
\begin{align*}
    \left[M\right]_{t\wedge\tau} = \int_0^{t\wedge\tau}\vert v_s\vert^2ds &\leq c\int_0^{t\wedge\tau} 3\left(\Gamma_0^{1/p}\right)^{3/2}ds \ \leq\ 3ce^{3/2}\,t = \tilde{c}_1\,t
\end{align*}
where in the first inequality we used Lemma \ref{lem:vt_control}. By Lemma \ref{lem:bernstein} we get :
\begin{align*}
     \PP(\tau\leq t) \leq \PP\left(M_{t\wedge\tau} \geq \Gamma_0^{1/p} \right) &= \PP\left(M_{t\wedge\tau} \geq \Gamma_0^{1/p} \ ,\ \left[M\right]_{t\wedge\tau} \leq \tilde{c}_1\,t \right) \\
     & \leq \exp\left(-\frac{\Gamma_0^{2/p}}{2\tilde{c}_1t}\right)\\
     & \leq \exp\left(-\frac{c_1}{t}\right).
\end{align*}
With $c_1 =\frac{1}{2\tilde{c}_1} $. Now notice that $3\Gamma_0^{1/p}\leq 3e\leq 10$ which proves the first statement. The second statement follows from Lemma \ref{lem:expansion}

\end{proof}
\begin{section}{Proof of the main theorem}

Take a subset $S$  of measure $1/2$ and $r>0$, for $t \leq  T = \frac{c_0}{\Psi_k^2\max(\log(k),1)}$ we have :

\begin{align*}
    \mu(S_r^c) = \E\mu_t(S_r^c) &\leq \E\left[\mu_t(S_r^c)\ind_{\mu_t(S)\geq\frac{1}{4}}\right] + \PP\left(\mu_t(S)\leq \frac{1}{4}\right)\\
    &\leq 4\exp(-\frac{1}{4} \min(\eta,t)r^2) +\PP(s_0 - s_t \geq \frac{1}{4} \ , \ [s]_t \leq \ 10t ) + \PP(\, [s]_{t} \geq \ 10t)\\
    & \leq 4\exp(-\frac{1}{4} \min(\eta,t)r^2) + \exp\left(-\frac{1}{320t}\right) + \exp\left(-\frac{c_1}{t}\right) \\
    &\leq 4\left(\exp(-\frac{1}{4} \min(\eta,t)r^2) + \exp\left(-\frac{c_4}{t}\right)\right)
\end{align*}
Where we used \eqref{gaussian_concentration_general} in the second line, Lemmas \ref{lem:bernstein} and \ref{lem:control_cov} in the third line and $c_4 = \min(c_1,\frac{1}{320})$ in the last line.

Define $\beta =\min(\eta,T)$ and choose $t(r) = \min(\eta,T,\frac{1}{r}) = \min(\beta,\frac{1}{r})$. We get that :

\begin{itemize}
    \item If $r\geq \frac{1}{\beta}$,
    \begin{equation}\label{eq:conc_big_r}
        \mu(S_r^c) \leq 8\exp(-c_5r)
    \end{equation}
    where $c_5 = \min(1/4,c_4)$
    \item If $r \leq \frac{1}{\beta}$,
        \begin{equation}\label{eq:conc_small_r}
        \mu(S_r^c) \leq 4\left(\exp\left(-\frac{1}{4} \min(\eta,T)r^2\right) + \exp\left(-\frac{c_4}{\min(\eta,T)}\right)\right) \leq 8\exp\left(-c_5 \beta r^2\right)
    \end{equation}
\end{itemize}
Overall this implies that for all $r>0$,
$$\mu(S_r^c) \lesssim \exp(-\min(c_0 r, c_1 \beta r^2))$$
which is the desired result.
\end{section}

\vspace{0.5cm}

\textbf{
Acknowledgements.} The author would like to thank the
anonymous reviewer for their careful reading of the manuscript as well as their suggestions on
presentation and writing.
%%  imsart-number.bst  will be used to                     %%
%%  MR numbers will be added by VTeX.                      %%
%%                                                         %%
%%  Use \cite{...} to cite references in text.             %%
%%                                                         %%
%%%%%%%%%%%%%%%%%%%%%%%%%%%%%%%%%%%%%%%%%%%%%%%%%%%%%%%%%%%%%

%% if your bibliography is in bibtex format, uncomment commands:
\bibliographystyle{plain}
\bibliography{biblio} 

\end{document}